\documentclass[12pt,a4paper]{amsart}  

\usepackage[utf8]{inputenc} 
\usepackage{xcolor}

\usepackage{array} 
\usepackage[colorlinks,linkcolor=blue]{hyperref} 
\usepackage{doi}

\def\mapright#1{\mathrel{%
\smash{\mathop{\longrightarrow}\limits^{#1}}}}
\renewcommand{\P}{\mathbb{P}}
\newcommand{\Z}{\mathbb{Z}}
\newcommand{\A}{\mathbb{A}}

\newcommand{\G}{\mathbb{G}}
\newcommand{\vp}{\varphi}
\newcommand{\ep}{\varepsilon}
\newcommand{\lra}{\longrightarrow}
\DeclareMathOperator{\Spec}{Spec}
\DeclareMathOperator{\Proj}{Proj}
\DeclareMathOperator{\res}{res}
\newcommand{\olsi}[1]{\,\overline{\!{#1}}}

\newtheorem{thm}{Theorem}[section]
\newtheorem*{thmstar}{Theorem}
\newtheorem{lemma}[thm]{Lemma}
\newtheorem{cor}[thm]{Corollary}
\newtheorem{prop}[thm]{Proposition}
\theoremstyle{definition}
\newtheorem{defn}[thm]{Definition}

\theoremstyle{remark}
\newtheorem{remark}[thm]{Remark}

\begin{document}
\title[Versal deformation 
of elliptic \textit{m}-fold points]{The versal deformation 
of elliptic \textit{m}-fold point curve singularities}

\author{Jan Stevens}
\address{\scriptsize Department of Mathematical Sciences, Chalmers University of
Technology and University of Gothenburg.
SE 412 96 Gothenburg, Sweden}
\email{stevens@chalmers.se}

\begin{abstract}
We give explicit, highly symmetric equations for the
 versal deformation of the singularity $L_{n+1}^n$ consisting
of $n+1$ lines through the origin in $\A^{n}(k)$ in generic position.  
These make evident that the base space 
of the versal deformation of $L_{n+1}^n$ is isomorphic to the
total space for $L_{n}^{n-1}$, if $n\geq 5$.
By induction it follows  that the base space
is irreducible and Gorenstein. 
We discuss the known connection to a modular compactification of the moduli space  
of $(n+1)$-pointed curves of genus 1.

For other elliptic partition curves it seems unfeasable to compute the versal deformation
in general. It is doubtful whether the base space is Gorenstein. For rational partition curves
we show that the base space in general has components of different dimensions.
\end{abstract}

\thanks{}
\maketitle

\section*{Introduction}
An explicit description of the versal deformation of a whole family of
singularities is only possible if the equations have a high degree of symmetry.
An example is the case of $n$ lines through the origin in the form 
of the coordinate axes in $\A^n$. The resulting equations are very simple. The ones given in \cite{Sch} are 
due to D. S. Rim, but the  
computation was done independently by various
authors \cite{Al,FP}. 
Here we consider the next case, 
the curve $L_{n+1}^n$
consisting of $n+1$  lines in generic position through the origin  in 
$\A^{n}(k)$, obtained by adding a line to the coordinate axes; in the terminology of
Smyth \cite{SmI} this an elliptic $m$-fold point, with $m=n+1$ the multiplicity. We assume that
$k$ is an algebraically closed 
field of characteristic zero.

As an elliptic $(n+1)$-fold point is quasihomogeneous, and has only deformations of negative
weight \cite{Pi}, its versal deformation can be fibrewise projectivised. Its base space $B_n$ is then
a fine moduli space for reduced projective  Gorenstein curves of arithmetic genus one
with a hyperplane section defined by a specific function $t$; this is a special case
of a general result of Looijenga \cite{Lo}. The projective scheme $\P(B_n)$ is a 
compactification  of the moduli space $M_{1,n+1}$ of $(n+1)$-pointed
curves of genus 1. It is isomorphic to the compactification 
$\olsi{M}_{1,n+1}(n)$ constructed by Smyth \cite{Sm}. This is proved by 
Lekili and Polishchuk \cite{LP}, who construct the projectivized family by computing the
coordinate ring of the fibres. In their equations the first two coordinates have a special role,
leading to a less symmetric form.
The connection to the versal deformation of $L_{n+1}^n$ is not made explicitly.

Our symmetric equations were originally obtained in a project with Theo de Jong to extend
the results of \cite{dJvS} on rational surface singularities with reduced fundamental cycle
(that is, with $L_n^n$ as general hyperplane section) to the case of minimally elliptic
singularities with reduced fundamental cycle.
The equations for $L_{n+1}^n$ and its deformations are very similar to those
for $L^n_n$ (the coordinate axes), but a bit more complicated. Whereas
for $L^n_n$ the equations are just $z_iz_j=0$ for $i<j$ we now get the equations
$z_iz_j=z_kz_l$, which as written do not provide a minimal system of generators of the ideal
of the curve. To obtain such a system the symmetry of the equations has to be broken.
Instead of using the full system of equations we introduce a new variable $y$
of  weight 2 and get $\binom n2$ equations $y=z_iz_j$. 
This simplifies the computation of the infinitesimal deformations.
For the versal deformation we end up with equations in terms of the original
full system. Due to the symmetry it suffices
to write down only one equation for the total space and one equation
for the base space.
The obtained equations are again similar to the equations for the versal deformation of the
coordinate axes $L_n^n$, but  a bit more involved. 

In the case at hand it is possible to analyse the equations, as they turn out to have
an inductive structure; in fact, due to the symmetry in several ways.
The base space 
of the versal deformation of $L_{n+1}^n$ is isomorphic to the
total space of the versal deformation of $L_{n}^{n-1}$, if $n\geq 5$. An equivalent
observation is made by Lekili and Polishchuk \cite{LP}, and used to study the base space.

\begin{thmstar}
The base space of the versal deformation of $L_{n+1}^n$ is Gorenstein and irreducible.
\end{thmstar}

The curve $L_{n+1}^n$ is the simplest elliptic partition curve 
of embedding dimension $n$. For all such curves the dimension 
of $T^1$ and $T^2$ is known \cite{BC}. The other extreme is the
Gorenstein monomial curve of minimal multiplicity, the
one with semigroup generated by $\langle n+1,n+2,\dots,2n\rangle$.
Its equations have less symmetry, and a general computation  of the 
versal deformation is 
out of reach. For $n=6$  the calculation is still doable, but the resulting equations
are too long to be given here. The situation can again be compared with
rational partition curves. For $L_n^n$ the base space is Cohen-Macaulay
\cite{PR}, with rather simple equations. They become more involved
if one starts from the  general hyperplane section of the cone over the rational normal 
curve of degree $n$ (with its standard toric equations). We prove here that the 
monomial curve with semigroup $\langle n, n+1,\dots,2n-1\rangle$ deforms into
non-smoothable singularities, if $n\geq 14$. In particular, its base space has components
of different dimensions. A similar elementary explicit deformation cannot be given for the
Gorenstein monomial curve, but we see no reason preventing the base space to be reducible.

\section{Gorenstein curve singularities with minimal $\delta$-invariant}
The singularity $L_{n+1}^n$  of $n+1$ lines through the origin in $\A^n$ 
in generic position is an elliptic partition curve \cite{BC}. We review the classification 
of these curves.

Let $(X,p)$ be a curve singularity with $r$ branches and let   
$\nu\colon \widetilde X \to X$ be the normalisation, with 
$\nu^{-1}(p)=\{p_1,\dots,p_r\}$. Denote the completion of the
local ring of $X$ at $p$ by $\widehat {\mathcal O}_X$ and the semi-local ring
$\oplus _{i=1}^r \widehat{\mathcal O}_{\widetilde X,p_i}$ by 
$\widehat {\mathcal O}_{\widetilde X}$. The \textit{$\delta$- invariant}
of $(X,p)$ is $\delta(X)=\dim_k \widehat {\mathcal O}_{\widetilde X}/
\widehat {\mathcal O}_X$.  For curves $X=X_1\cup X_2\subset \A^N$ with singular point 
$p$,
where the $X_i$ may be reducible,
$\delta (X) = \delta(X_1)+\delta(X_2)+(X_1\cdot X_2)$; here the intersection multiplicity
$(X_1\cdot X_2)$ is 
$\dim \widehat {\mathcal O}_{\A^N}/(\widehat I_1+\widehat I_2)$ with 
$\widehat I_1$ and $\widehat I_2$ the ideals of $X_1$ and $X_2$ in the completed
local ring of $\A^N$ at $p$. 

The Milnor number $\mu(X)$  has been defined by Buchweitz and Greuel and it is equal to 
$2\delta-r+1$ \cite[Prop. 1.2.1]{BG}. For a smoothable curve singularity
over the complex numbers the Milnor number is equal to the first Betti number of
the Milnor fibre \cite[Cor. 4.2.3]{BG}, and the genus of the Milnor fibre
is equal to $\delta-r+1$. Therefore we define for all curve singularities the genus
as $g(X)=\delta-r+1$. We have 
\begin{equation}\label{eq:genus-of-union}
g(X_1\cup X_2)=g(X_1)+g(X_2)+(X_1\cdot X_2)-1\;.
\end{equation}

A curve singularity $X$ is \textit{decomposable} (into
$X_1$ and $X_2$), if $X$ is the union of two curves $X_1$ and $X_2$, which 
lie in smooth spaces intersecting each other transversally in the
singular point of $X$. We write $X=X_1 \vee X_2$.
With the exception of an ordinary double point 
a decomposable curve  is never Gorenstein.

Curves singularities with minimal $\delta$-invariant occur in the literature
under several names (see e.g. \cite{St}). Here we call them partition curves
following \cite{BC1}.  They are wedges of monomial curves. Let $X_n$ be the
monomial curve with semigroup generated by $\langle n,n+1,\dots,2n-1\rangle$.
For a partition $\boldsymbol{p}=(n_1,\dots,n_r)$ of $n$ we define
\[
X_{\boldsymbol{p}} = X_{n_1}\vee \dots \vee X_{n_r}\;.
\]
In particular $X_{(n)}=X_n$, and $X_{(1,\dots,1)}=L_n^n$, the singularity
consisting of the coordinate axes. We include the smooth point $L_1^1$, which has
$\delta= 0$.
Partition curves are exactly the curves of multiplicity $m$ equal to the
embedding dimension $n$ with $\delta = m-1 = n-1$. Given $n$, they have  
minimal multiplicity and minimal $\delta$. They occur as hyperplane sections of rational singularities.

For an isolated Gorenstein curve singularity $Y$ of embedding dimension
$n$,  $n\geq3$, the delta invariant has value at least $n + 1$. 
Gorenstein curve singularities  with $\delta= n + 1$, $n\geq 2$, are classified
in \cite{BC} and are called elliptic partition curves. The term elliptic in the name is explained by the fact
that such curves occur as hyperplane sections of minimally elliptic
singularities.
The easiest description is as follows: given a partition 
$\boldsymbol{p}=(n_1,\dots,n_r)$ of $n+1$, the elliptic partition curve
$Y_{\boldsymbol{p}} \subset \A^n$ is the curve obtained by a generic linear
projection of  the partition curve
$X_{\boldsymbol{p}}   \subset \A^{n+1}$.  In particular $Y_{(n+1)}$ is the
monomial curve generated by  $\langle n+1,n+2,\dots,2n\rangle$, and 
$Y_{(1,\dots,1)}=L^n_{n+1}$,  the curve consisting of $n+1$ lines through
the origin in generic position. The generic projection might not respect the grading,
but elliptic partition curves have no moduli, and there exists a
quasi-homogeneous representative. 

\begin{remark}
The elliptic partition curves (with $n\geq 2$) have also minimal multiplicity $m=n+1$.
So they are the curves with minimal $\delta=m$. 
This view point allows to extend the classification to $n+1=m=2$. There are two Gorenstein
curves with $m=\delta=2$: the tacnode $A_3$ corresponding to the partition $(1,1)$
and the ramphoid cusp $A_4$ corresponding to $(2)$. 
Because $\dim \mathcal{O}_X/\mathcal{O}_Y = 1$ for the generic projection $X\to Y$,
 the curves $A_3$ and $A_4$ can be
considered to be the projections of $A_1$ and $A_2$. We remark that $A_1$ and $A_2$
are the first blow-ups of $A_3$ and $A_4$. The first blow-up of $A_2$ is a smooth curve,
so $A_2$ can be considered to belong to the partition $(1)$, for $n=0$.
With these conventions \cite[Prop. 4.1.1]{BC} continues to hold for $n=0,1$:
the  general hypersurface section
of a minimally elliptic surface singularity with $-Z\cdot Z=n+1$ (where $Z$ is the 
fundamental cycle) is an elliptic partition curve for a partition of $n+1$.
\end{remark}

The curves $L_{n+1}^n$ are also elliptic in the sense that they satisfy
$g=\delta-r+1=1$, so $\delta=r$. The classification of curves singularities
with $g=1$ is due to Greuel \cite{Gr0,Gr}. By \eqref{eq:genus-of-union}
such a curve is of the form $X\vee L_s^s$ with $X$ indecomposable of genus 1.
For indecomposable $X$ with $r\geq 2$  all branches are smooth
and by removing one branch a curve  $L_{r-1}^{r-1}$ is left. Therefore
such an $X$ is an elliptic partition curve belonging to the partition $(1,\dots,1)$ of $n+1$,
including the cases $n=0,1$, and consequently it is Gorenstein.  
These curves are called elliptic $m$-fold points by Smyth \cite{SmI}, who comes to the
same classification in arbitrary characteristic.

Behnke and Christophersen have determined the dimension of $T^1$ and $T^2$
for all elliptic partition curves \cite[Prop. 5.4.1]{BC}:
\begin{prop}
 For an elliptic partition curve $Y$ of multiplicity $n + 1$, where $n\geq 4$, with $r$ branches
 \begin{align*}
\dim T^1_Y &= \frac{n(n+1)}2-r+1\;, \\
\dim T^2 &= \frac{n(n+1)(n-4)}6\;.
\end{align*}
Furthermore $T^2_Y$ is annihilated by the maximal
ideal of the local ring.
\end{prop}

For $n=2$ the curve is a plane curve and for $n=3$ it is a complete intersection,
so the formulas do not extend to these cases.

For the special case of  $L_{n+1}^n$ it had already been shown 
by Pinkham \cite{Pi} that $T^1$ is negatively graded, and Greuel \cite{Gr}
had computed  its dimension:  $T^1$ is concentrated in degree $-1$, if $n\geq 4$,
and has dimension $\binom n2$. 

\section{Deformations of negative weight}
By a result of Pinkham \cite{Pi} the curve  $L_{n+1}^n$ has only deformations of negative weight.  
This allows to  take 
the projective closure of the fibres of the versal family  and obtain  in this way 
also the versal deformation of the projective cone over
$n+1$ points in $\P^{n-1}$ in general position. 
We recall the general construction, following Looijenga, who makes Pinkham's results
\cite{Pi,Pi2}
more precise  in the
Appendix of \cite{Lo}.

Let $X\subset \A^n$ be an isolated  quasi-homogeneous singularity, 
over a fixed algebraically closed field  $k$ of characteristic 0.
This means that the ideal of $X$ is generated by quasi-homogeneous
polynomials. Then 
$X=\Spec R$, where $R=\oplus _{l=0}^\infty R_l$ is a 
reduced $\Z_+$-graded $k$-algebra, with $R_0=k$. The grading on $R$ corresponds
to a $\G_m$-action on $X$, defined by $\lambda \cdot \vp = \lambda^l \vp$ for 
$\lambda \in \G_m$ and $\vp\in R_l$. As $R_0=k$ and $R_l=0$ for $l<0$, this action is good,
meaning that the unique fixed point, defined by the maximal ideal 
$R_+=\oplus _{l=1}^\infty R_l$, is in the closure of every orbit.
The standard projectivisation  $X\subset \olsi X$ is defined as follows:
let $\olsi R_k=R_0\oplus\dots\oplus R_k$, then $\olsi R:= 
\oplus _{k=0}^\infty \olsi R_k$ is a graded $\Z_+$-algebra and 
$\olsi X = \Proj \olsi R$.  If $t=(1,0)\in \olsi R_1=
R_0\oplus R_1$, then $\olsi R = R[t]$  and $X$ becomes a subscheme of $\olsi X$ by
making $t$ invertible. The complement $\Proj R=  \olsi X_\infty=\olsi X \setminus X $ 
is the divisor defined by $t$.

A deformation $(\pi,i)$, where  $\pi \colon (\mathcal X,X_s)\to (S,s)$ and $i\colon X\cong X_s$,
has (good) $\G_m$-action if both  $(\mathcal X,X_s)$ and $(S,s)$ have 
(good) $\G_m$-actions making  $\pi$ and $i$  $\G_m$-equivariant. In case of a good 
 $\G_m$-action the (formal) schemes $(\mathcal X,X_s)$ and $(S,s)$ can be taken to be
 affine schemes $\mathcal X = \Spec \mathcal R$ and $S=\Spec A$. 
 By the same construction as for $X$ the deformation can be fibrewise projectivised.
Put $\olsi{\mathcal R}_k= A \mathcal R_0\oplus \dots
 A \mathcal R_k$ and define $\olsi{\mathcal X}=\Proj_A\olsi{\mathcal R}$.
 If $t=(1,0)\in \olsi{\mathcal R}_1$ then $\olsi{\mathcal R}/t\olsi{\mathcal R}$
 is naturally isomorphic to $A\otimes_k R$ and $\olsi{\mathcal X}_\infty$ is
 naturally isomorphic to $\olsi{X}_\infty\times S$. 
 
 By \cite{Pi2} there exists a miniversal object
 $(\pi,i)$ for deformations of $X$ with  $\G_m$-action, which also induces
 a miniversal deformation of the isolated singularity $(X,p)$. The part $\pi_-$
of $\pi$ of negative weight is miniversal for deformations of $X$
with good  $\G_m$-action. Looijenga proves that $\pi_-$
is actually universal for this property \cite[Thm A.2]{Lo}:

\begin{thm}
Let $X$ be a reduced affine scheme with good $\G_m$-action with as only singular point
the vertex. The negative weight part $\pi_-$ of the miniversal deformation of $X$
is a final object in the category of deformations with good $\G_m$-action. The group 
$G$ of automorphisms of $X$ commuting with the  $\G_m$-action acts on  $\pi_-$ 
and its  projectivisation $\bar{\pi}_-$. Any isomorphism between 
two fibres of the $\bar{\pi}_-$ preserving  the
pieces  at  infinity is induced by a unique element of $G$.
\end{thm}

If the isomorphism $(\bar\sigma, \bar\sigma_\infty)$, induced by $g\in G$,
satisfies $\bar\sigma_\infty=1$, then $g\in \G_m$ \cite[Lemma A.4]{Lo}.


The morphism $\bar\pi_-$ has also a moduli interpretation. Looijenga defines an $R$-polarised
scheme as a triple $(\olsi Z, \olsi Z_\infty,\vp^*)$ consisting of a projective scheme
$\olsi Z$, an ample reduced Weil divisor $\olsi Z_\infty$ on $\olsi Z$ and an isomorphism
$\vp^*\colon R_{\olsi Z}/t  R_{\olsi Z} \to R$ of graded $k$-algebras, where
$R_{\olsi Z} =\bigoplus_{l=0}^\infty H^0(\mathcal O_{\olsi Z}(l \olsi Z_\infty))$ with 
$t\in  H^0(\mathcal O_{\olsi Z}(\olsi Z_\infty))$ the element corresponding to $1$ \cite[A.5]{Lo}.
He shows:

\begin{prop}
The morphism 
$\bar{\pi}_-\colon (\olsi {\mathcal X},\olsi{\mathcal X}_\infty)\to S_-$
is a fine moduli space for $R$-polarised schemes.
\end{prop}

We specialise to the case  $L_{n+1}^n$. This curve has only deformations of negative weight, so
 $\pi_-=\pi$. Eech  geometric fibre
of $\bar{\pi}$  is a reduced  Gorenstein curve of arithmetic genus
1 with $n+1$ marked points (the points at infinity). 

The approach of Lekili  and Polishchuk \cite{LP} to determining explicit equations for $\bar{\pi}$ 
in this case is to construct the ring $R_{\olsi Z}$. They start from 
 a reduced, connected projective curve $C$ of arithmetic genus 1 over an arbitrary 
 algebraically closed field 
with $n+1$ distinct smooth marked points $p_0,\dots,p_n$ such that 
$\mathcal O_C(p_0+\dots+p_n)$ is ample and $h^0(\mathcal O_C(p_i))=1$ for all $i$.
Generators of $R_C=\bigoplus_{l=0}^\infty H^0(\mathcal O_{C}(l(p_0+\dots+p_n))$
are $1$ and functions $h_{0i}\in H^0(\mathcal O_{C}(p_0+p_i))$ with 
$-\res_{p_i}(h_{0i}\omega)=\res_{p_0}(h_{0i}\omega)=1$ for a fixed generator $\omega$ of
$H^0(C,\omega_C)$.  Under certain normalisations they determine the ring structure.
In their equations the indices $1$ and $2$ have a special role.
This approach does not extend to the computation of the versal deformation of other 
elliptic partition curves.

\section{Computation of the versal deformation}
In this section we compute the versal deformation for $L_{n+1}^n$ for $n\geq 4$
using equations and relations.

\subsection{}
For Gorenstein singularities of minimal multiplicity $m$ it is known \cite{Wa}
in general
that the local ring has a free resolution with Betti numbers
\[
\beta_i = \frac{i(m-2-i)}{m-1} \binom{m}{i+1},
\qquad i=1,\dots,m-3\;,
\]
while $\beta_0= \beta_{m-2}=1$.

We consider $L_{n+1}^n$, with multiplicity $m=n+1$.
We can take $n$ of the lines to be the coordinate axes and as last line the
line through $(1,\dots,1)$. Then the ideal $I_0\subset P_0=k[z_1,\dots,z_n]$
of $C_n$ is minimally generated by $\frac{(n+1)(n-2)}2=\binom{n}2-1$ polynomials. 
We start from
a non-minimal system of generators
\[
f_{i,j;k,l} = z_iz_j-z_kz_l, \qquad i\neq j, k\neq l, \quad 1\leq i,j,k,l\leq n\;.
\]
A minimal system
can be chosen to consist  of the $F_{ij;1,2}=z_iz_j-z_1z_2$, $i<j$, 
but this choice gives the first two 
variables a special role.

To write symmetric equations and at the same time minimize the number
of equations it is convenient to embed the singularity in $\A^{n+1}$. We introduce
a new variable $y$ of weight 2 with $\A^{n}=\{y=0\}$ and after a
coordinate transformation, say replacing $y$ by $y-z_1z_2$, which transforms
the line through $(1,\dots,1)$ into a parabola, the ideal $I$ of the curve $C_n$
in $\A^{n+1}$ is generated by
the $\binom{n}2$ polynomials
\[
g_{ij}=z_iz_j-y, \qquad 1\leq i<j\leq n \;.
\]
Next we describe the relations between these generators.
Write $P=k[z_1,\dots,z_n,y]$ for the (graded) polynomial ring in $n+1$ variables.
Denote by $I_0$ the ideal in $P$ generated by the polynomials $f_{i,j;k,l} $.
There is an exact sequence
\[
0 \lra P/I_0 \mapright f P/I_0 \lra P/I \lra 0
\]
where $f=z_1z_2-y$.
Let $F_0$ be  a free resolution of $P_0/I_0$. The same matrices yield a free resolution $F$
of $P/I_0$.   One obtains a resolution of $P/I$
as a mapping cone of a homomorphism of complexes $F\to F$ which is a lift of
$f\colon P/I_0 \to P/I_0$. In particular the number of relations is
$(n+1)(n-2)/2+(n^2-1)(n-3)/3$. The first summand gives the number of 
Koszul relations, which we can ignore in deformation computations. The other 
relations are linear relations, that we now describe.

We start by forming $(n-2)\binom n2=3\binom n3$ expressions of the form $z_ig_{jk}$.
There are $\binom n3$ monomials $z_iz_jz_k$. As 
$z_ig_{jk}-z_jg_{ik}=(z_i-z_j)y$ and $z_ig_{jk}-z_kg_{ij}=(z_i-z_k)y$ we get
$n-1$ additional conditions on linear combinations of the $z_ig_{jk}$ to be
syzygies. This gives the required number of $2\binom n3 -(n-1)=(n^2-1)(n-3)/3$
linear independent relations of the form
\[
R_{ik;jl}= z_k(g_{ij}-g_{il})-z_i(g_{kj}-g_{kl})\;.
\]
This computation also shows that the relations between the polynomials
$f_{i,j;k,l}$ are generated by the linear dependencies and
\[
R_{ik;jl}= z_kf_{ij;il}-z_if_{kj;kl}\;.
\]

\subsection{}
We determine an explicit basis for  the module of infinitesimal deformations   
$T^1$, which by \cite{Gr} is concentrated in degree $-1$. 
We perturb the generators $g_{ij}$ to
\[
G_{ij}=z_iz_j-y+\sum a_{ij}^mz_m
\]
and insert these expressions in the relation $R_{ik;jl}$:
\begin{multline*}
z_k(G_{ij}-G_{il})-z_i(G_{kj}-G_{kl})\\=
z_k\left(\sum a_{ij}^mz_m -\sum a_{il}^mz_m\right)
      -z_i\left(\sum a_{kj}^mz_m -\sum a_{kl}^mz_m\right)\;.
\end{multline*}
The coefficient of $z_k^2$ is $a_{ij}^k- a_{il}^k$ and has to vanish. From this it
follows that $a_{ij}^k=a^k$ for $k\neq i,j$. Then 
$z_k(G_{ij}-G_{il})-z_i(G_{kj}-G_{kl})$ is modulo $I$ equal to
\begin{equation}\label{y-coef}
y\left((a_{ij}^i-a_{il}^i)-(a_{kj}^k-a_{kl}^k)+(a_{ij}^j-a_{kj}^j)-(a_{il}^l-a_{kl}^l)\right)
\end{equation}
and the coefficient of $y$ has to vanish.
We put 
\[
b_{ij}=a_{ij}^i+a_{ij}^j\;.
\]
Then the condition that \eqref{y-coef} vanishes can be written as
\begin{equation}\label{b-eqn}
b_{ij}-b_{il}=b_{kj}-b_{kl}\;.
\end{equation}
This system of equations allows all $b_{ij}$ to be expressed in terms 
of $n$ of them, say $b_{1k}$, $2\leq k \leq n$ and $b_{23}$.  A more symmetric
solution is to introduce variables $b_i$, $1\leq i \leq n$, and put
\[
b_{ij} = b_i+b_j\;.
\]
This solves the equations \eqref{b-eqn} and the $b_i$ are determined by the
$b_{ij}$:
we have $b_{12}=b_1+b_2$, $b_{13}=b_1+b_3$ and $b_{23}=b_2+b_3$ so
\begin{align*}
2b_1&=b_{12}+b_{13}-b_{23}\\
2b_2&=b_{12}+b_{23}-b_{13}\\
2b_3&=b_{13}+b_{23}-b_{12}
\end{align*}
and
\[
b_k=b_{1k}-b_1=b_{1k}+\textstyle{\frac12 b_{23}-\frac12 b_{13}-\frac12 b_{12}}\;.
\]
We apply coordinate transformations to the
\[
G_{ij}=z_iz_j-y+ a_{ij}^iz_i+ a_{ij}^jz_j+ \sum_{s\neq i,j} a^sz_s
\]
by transforming 
$y\mapsto y+\sum_{s}a^s z_s$ and 
$ z_i\mapsto z_i+a^i-b_i$.
Taking only first order terms into account the result is
\[
G_{ij}=z_iz_j-y+ (a_{ij}^i-a^i+a^j-b_j)z_i+ (a_{ij}^j-a^j+a^i-b_i)z_j \;.
\]
The coefficients of $z_i$ and $z_j$ sum to zero:
\[
(a_{ij}^i-a^i+a^j-b_j)+ (a_{ij}^j-a^j+a^i-b_i)=
a_{ij}^i+ a_{ij}^j-b_i-b_j=0\;.
\]
Finally we put $a_{ij}=a_{ij}^i-a^i+a^j-b_j$ and
$a_{ji}=a_{ij}^j-a^j+a^i-b_i$. Then $a_{ij}+a_{ji}=0$.
As deformation variables we can take the $a_{ij}$ with $i<j$, 
but  it will
be convenient to allow also $a_{ij}$ with $i>j$, and we 
will freely use that $a_{ij}=-a_{ji}$ . The result is the following.

\begin{lemma}
A basis for $T^1$ is represented by the $\binom n2$ first order defomations
\[G_{ij}=z_iz_j-y+ a_{ij}(z_i-z_j)\]
with $i<j$.
\end{lemma}


\subsection{}
To compute the versal deformation we have to lift the relations $R_{ik;jl}$.
We 
compute a lift of the
relation up to first order:
\begin{multline*}
z_k(G_{ij}-G_{il})-z_i(G_{kj}-G_{kl})\\-a_{ij}(G_{ik}-G_{jk})+a_{il}(G_{ik}-G_{lk})
+a_{kj}(G_{ik}-G_{ij})-a_{kl}(G_{ik}-G_{il})
\\{}=
z_k(a_{ij}z_i+a_{ji}z_j -a_{il}z_i+a_{li}z_l)-z_i(a_{kj}z_k+a_{jk}z_j -a_{kl}z_k+a_{lk}z_l)\\
\begin{aligned}
&-a_{ij}(z_iz_k-z_jz_k+a_{ik}z_i+a_{ki}z_k-a_{jk}z_j-a_{kj}z_k)\\
&+a_{il}(z_iz_k-z_lz_k+a_{ik}z_i+a_{ki}z_k-a_{lk}z_l-a_{kl}z_k)\\
&+a_{kj}(z_iz_k-z_iz_j+a_{ik}z_i+a_{ki}z_k-a_{ij}z_i-a_{ji}z_j)\\
&-a_{kl}(z_iz_k-z_iz_l+a_{ik}z_i+a_{ki}z_k-a_{il}z_i-a_{li}z_l)
\end{aligned}\\{}
=(z_k-z_i)(a_{ij}a_{ik}+a_{kj}a_{ki}+a_{jk}a_{ji}-a_{il}a_{ik}-a_{kl}a_{ki}-a_{lk}a_{li})\;.
\end{multline*}
We abbreviate 
\[
\vp_{ijk} = a_{ij}a_{ik}+a_{ji}a_{jk}+a_{ki}a_{kj}
\]
so that the term in parentheses on the right hand side of the equation above
becomes $\vp_{ijk}-\vp_{ilk}$. The right hand side can be made to vanish
by subtracting this term from $F_{ij;il}=G_{ij}-G_{il}$, for all choices of three out of 
the four indices $i$, $j$, $k$ and $l$; this means
subtracting $\vp_{kji}-\vp_{kli}=\vp_{ijk}-\vp_{ilk}$ from $F_{kj;kl}=G_{kj}-G_{kl}$.
To check that this indeed gives a lift of the relation we have to compute that
\[
-a_{ij}(\vp_{ikl}-\vp_{jkl})+a_{il}(\vp_{ijk}-\vp_{jlk})
+a_{kj}(\vp_{ikl}-\vp_{ijl})-a_{kl}(\vp_{ijk}-\vp_{ijl})=0\;.
\]
For this it suffices to find for every term in the first summand a term which cancels it;
the computation is  not difficult.

We have now lifted the relation $R_{ik;jl}$, but to lift the relation $R_{im;jl}$ for $m\neq k$
we need that   $\vp_{ijk}-\vp_{ilk}=\vp_{ijm}-\vp_{ilm}$. These equations describe
the base space. We formulate  the result of our computation directly
in terms of the $a_{ij}$.

\begin{thm}\label{versal-def-thm}
The versal deformation of $L_{n+1}^n$, where $n\geq 4$, is given by the polynomials
\begin{align*}
F_{ij;il}&= z_iz_j-z_iz_l+(a_{ij}-a_{il})z_i+a_{ji}z_j-a_{li}z_l \\
&\qquad - a_{ij}a_{ik}-a_{ji}a_{jk}-a_{ki}a_{kj}+ a_{il}a_{ik}+a_{li}a_{lk}+a_{ki}a_{kl}\\
&=(z_i-a_{ij})(z_j-a_{ji})-(a_{ik}-a_{ij})(a_{jk}-a_{ji})\\
&\qquad -(z_i-a_{il})(z_l-a_{li})+(a_{ik}-a_{il})(a_{lk}-a_{li})\;,
\end{align*}
where $a_{ij}+a_{ji}=0$ for all $1\leq i,j\leq n$,
with base space $B_n$ given by 
\begin{multline*}
\quad \Phi^i_{jl;km}=(a_{ik}-a_{ij})(a_{jk}-a_{ji})-(a_{ik}-a_{il})(a_{lk}-a_{li})\\
- (a_{im}-a_{ij})(a_{jm}-a_{ji})+(a_{im}-a_{il})(a_{lm}-a_{li})\quad
\end{multline*}
for all distinct $i$, $j$, $k$, $l$, $m$.
\end{thm}  

It can be checked that the formulas (1.1) -- (1.4) in \cite{LP} only differ from the above ones
by a coordinate transformation. Their formulas are less symmetric, as they choose to give the
first two coordinates a special role.

Only in the case $n=4$, where there are no equations for the base space, which 
therefore is smooth, there are nice formulas for the total
space in terms of the $G_{ij}$:
\[
G_{ij}=z_iz_j-y+ a_{ij}z_i+a_{ji}z_j-\vp_{ijk}-\vp_{ijl}\;.
\]
For general $n$  one can take the same formulas for $1\leq i < j \leq 4$ and then 
find the other $G_{kl}$ as all $G_{ij}-G_{il}$ are known.

\subsection{}\label{sect_smooth}
It is well known that the curve $L_{n+1}^n$ is smoothable. 
One reason is that it is a general hyperplane section of the cone over an
elliptic normal  curve of degree $n+1$, so \lq sweeping out the cone\rq\ defines
a smoothing \cite[(7.6)]{Pi}.
Pinkham also describes an inductive
procedure, where two lines are smoothed to a quadric with given tangent direction,
forming an $L_{n}^{n-1}$
\cite[(11.13)]{Pi}. Such a partial smoothing occurs along the 
$a_{ij}$-axis in the base space. To be specific, we choose
$a_{n-1,n}=t$ and $a_{ij}=0$ for $(i,j)\neq (n-1,n)$. This solves the equations for the
base space as $\vp_{ijk}=0$ for all $(i,j,k)$. The equations for the total space of this
1-parameter deformation are
\[
z_iz_j-y=0,  \qquad z_{n-1}z_n-y+t (z_{n-1}-z_n)=0\;.
\]
For $t\neq0$ this is a curve consisting of  $n$ smooth branches.
Among  these are  the  $z_i$-axes for $1\leq i\leq n-2$ and the curve $(t,\dots,t,t^2)$.
Together they form a curve consisting of  $n-1$ smooth branches in general position
and the linear space spanned by their tangents at the origin  intersects the 
$(z_{n-1},z_n)$-plane in the diagonal.  The last branch of the deformed curve
is a hyperbola in the 
$(z_{n-1},z_n)$-plane with the diagonal as tangent line at the origin. This is
indeed a singularity of type $L_{n}^{n-1}$.

It is also possible to smooth one coordinate axis and the parabola with tangent
through $(1,\dots,1,0)$ (in $(z_i,y)$-coordinates).
If the axis is the $z_n$-axis, then we  have the following deformation:
\[
\begin{cases}
(z_{i}-t)(z_n+t)-y=0, & 1\leq i \leq n-1,\\
 z_iz_j-y=0,  &1\leq i < j\leq n-1.
\end{cases}
\]

 \section{The base space}
 
 \subsection{}
The base space is smooth for $n=4$. In that case the curve is codimension three Gorenstein
so it and its deformations can be given as Pfaffians of a skew $5\times5$ matrix.
For $n\geq5$ we have to analyse the polynomials
\[
\Phi^i_{jl;km}=\vp_{ijk}-\vp_{ilk}-\vp_{ijm}+\vp_{ilm}\;.
\]
We note that $\Phi^i_{jl;km}$ is antisymmetric in $j,l$ and $k,m$ and symmetric in the
pairs $(j,l),(k,m)$. Furthermore
\begin{equation}\label{phi-eqn}
\Phi^i_{jl;km}-\Phi^n_{jl;km}-\Phi^k_{jl;in}+\Phi^m_{jl;in}=0\;.
\end{equation}

Without explicitly identifying their equations as the versal deformation of $L_{n+1}^n$,
Lekili and Polishchuk observe that the total space of the family for $n-1$
is isomorphic to the base space for $n$ \cite[Prop.~1.1.5 (ii)]{LP}. 
This can be seen easily from our equations.

\begin{thm}\label{thm:basetotal}
For $n\geq 5$ 
the base space $B_n$ of the versal deformation of $L_{n+1}^n$
is  isomorphic to 
the total space of the versal deformation of $L_{n}^{n-1}$.
\end{thm}
\begin{proof}
We start from the equations of the total space and the base space
in Theorem \ref{versal-def-thm}.
By substituting $z_i=a_{in}$ the polynomial $F_{ij;il}$ (written with the index $k$)
becomes  the polynomial $\Phi^i_{jl;nk}$. 
We establish by induction that this procedure gives (together with the equations for the base space of $L_{n}^{n-1}$) equations describing the base space of $L_{n+1}^n$,
by showing that the number of independent equations  is equal to the
dimension of $T^2$, which is $\frac{n(n+1)(n-4)}6$.

The polynomials $F_{ij;il}$  give
$\frac{n(n-3)}2$ linearly independent equations.
The polynomials  $\Phi^i_{jl;mk}$ not involving the index $n$ describe
the base space $B_{n-1}$, and  by the induction hypothesis $\frac{n(n-1)(n-5)}6$  of them are linearly independent (the base case is $n=5$, where no such polynomials exist).
This gives   $\frac{n(n-3)}2+
\frac{n(n-1)(n-5)}6= \frac{n(n+1)(n-4)}6$ linear independent quadratic
equations for $B_n$. 
We remark that this procedure does not give polynomials of the form  $\Phi^n_{jl;km}$,
but equation \eqref{phi-eqn} shows that such polynomials are linear combinations
of the other ones. Therefore we obtain all equations of  Theorem \ref{versal-def-thm}
for $B_n$.
\end{proof}

\begin{cor}
The base space $B_n$ is Gorenstein, has dimension $n+2$ and multiplicity  $\frac{n!}{24}$, and is smooth in codimension 6.
\end{cor}
\begin{proof}
By induction on $n$. For $n=4$ the result holds, as $B_4$ is smooth.
We view $B_n$, $n\geq5$, as total space of the versal deformation  of $L_{n}^{n-1}$. By homogeneity it suffices to show the Gorenstein property 
for the local ring at the origin. For a flat local morphism $\vp \colon A\to B$ of 
local Noetherian rings $B$ is Gorenstein if and only if $A$ and $B/\mathfrak{m}_AB$
both are Gorenstein \cite[Thm. 23.4]{Mat}.	As both the base and the special fibre
are Gorenstein, the same therefore holds for the total space.

The total space is Gorenstein and therefore Cohen-Macaulay, so all irreducible components have the same dimension, given by
the formula for the dimension
of  a smoothing component, which yields $n+2$. The dimension statement follows also by induction, the base being that 
the dimension of $B_4$ is 6. 
The curve $L_{n+1}^n$ deforms (over reduced bases)
only to other elliptic $m$-fold points or ordinary double points \cite[Prop. 3.6]{Gr}.  As $B_4$ is smooth for $n\leq 4$, it follows that $B_n$ is smooth in codimension 6.
Therefore $B_n$ is normal, and there is only one component.

The multiplicity follows again by induction. We can also use the formula for the Hilbert series of the graded ring in \cite[Cor. 1.1.7]{LP}
\end{proof}

The inductive construction of the base space makes it possible to give a minimal
system of equations. 
%
There are $\binom n3$ expressions $\vp_{ijk}$ and the
$ \frac{n(n+1)(n-4)}6$ equations allow to express exactly so many in terms of
$n$ of them, for which we take the four $\vp_{ijk}$ with $1\leq i < j < k \leq4$
and the $n-4$ expressions $\vp_{12k}$ with $5\leq k \leq n$.

\begin{prop}\label{base-prop}
A minimal system of equations for the base space $B_n$  is
\begin{align*}
\vp_{1ij} & = \vp_{12i}+\vp_{12j}+\vp_{134}-\vp_{123}-\vp_{124} \\
\vp_{2ij} & = \vp_{12i}+\vp_{12j}+\vp_{234}-\vp_{123}-\vp_{124} \\
\vp_{ijk} & = \vp_{12i}+\vp_{12j}+\vp_{12k}+\vp_{134}+  
                          \vp_{234}-2\vp_{123}-2\vp_{124} 
\end{align*}
where $3\leq i<j\leq n$, $(i,j)\neq (3,4)$ in the first two lines and 
$3\leq i<j<k\leq n$.
\end{prop}

\begin{proof}
We reduce $\vp_{ijk}$ using 
$\Phi^k_{i1;j2}=\vp_{ijk}-\vp_{1jk}-\vp_{2ik}+\vp_{12k}$ for $i\geq 3$ 
and $\Phi^1_{i2;j3}=\vp_{1ij}-\vp_{12j}-\vp_{13i}+\vp_{123}$,  $\Phi^2_{1;j3}$,  and
finally $\Phi^1_{32;i4}=\vp_{13i}-\vp_{12i}-\vp_{134}+\vp_{124}$ and $\Phi^2_{31;i4}$. 
In the last two expressions all terms cancel
if  $i=4$. Therefore the formulas in the statement continue to hold if an index has the value
$3$ or $4$; some terms then cancel.
\end{proof}

Conversely, we also get equations for the total space, which deform the minimal
generating set $z_iz_j-z_1z_2$,
by substituting  $a_{i,n+1}=z_i$ in the equations  for $B_{n+1}$ just given.
This gives  rather complicated formulas, as the general expression for 
$\vp_{ijk}$ shows.

For $n=5$ the resulting equations of $B_5$  can be written as the Pfaffians
of the skew symmetric $5\times5$ matrix
\[\small
\begin{bmatrix}
a_{24}-a_{25} &
-a_{14}+a_{15} &
-a_{23}+a_{25} &
a_{13}-a_{15}\\
& a_{34}-a_{35}&
-a_{12}-a_{25}&
a_{12}-a_{13}+a_{24}-a_{34}\\
&&
a_{12}-a_{14}+a_{23}+a_{34}&
-a_{12}+a_{15}\\
&&&
a_{34}+a_{45}
\end{bmatrix}
\]
where we only write the part of the matrix above the diagonal.

\subsection{}
The interpretation of the base space $B_n$  as fine moduli space for $R$-polarised 
curves shows that 
the projectivisation $\P(B_n)$ is 
a compactification of the moduli space $M_{1,n+1}$ of $(n+1)$-pointed
curves of genus 1 by  curves with Gorenstein singularities; by forgetting the choice of $t$
all curves $(C,p_0,\dots,p_n)$ above a line in the base space are isomorphic.

We compare this compactification with the compactifications constructed by
Smyth \cite{SmI, Sm}.

\begin{defn}\label{def:smyth}
Let $C$ be a connected, reduced, complete curve of arithmetic genus one, with
$n+1$ distinct smooth marked points $p_0,\dots,p_n$. Let $m<n+1$. 
The curve $(C,p_0,\dots,p_n)$ is \textsl{$m$-stable} if
\begin{enumerate}
\item the curve $C$ has only nodes and Gorenstein singularities of genus one
with $r\leq m$ branches as singularities,
\item if $E\subset C$ is any connected subcurve with $p_a=1$, then
\[
|E\cap \overline{C\setminus E}|+|\{p_i\in E\}| > m\;,
\]
\item $H^0(C,\Omega^\vee_C(-\sum p_i))=0$.
\end{enumerate}
\end{defn}

Smyth proves that the moduli stack $\olsi{\mathcal M}_{1,n+1}(m)$ of $m$-stable curves is
a proper irreducible Deligne-Mumford stack over $\Spec \Z[1/6]$. 
In \cite{Sm}, working over a fixed algebraically closed field of characteristic zero, he proves 
that the corresponding coarse moduli space $\olsi{M}_{1,n+1}(m)$ 
is projective.

\begin{prop}
The moduli space  $\olsi{M}_{1,n+1}(n)$ is isomorphic to the
projectivisation $\P(B_n)$ of the base space of the versal deformation of $L^n_{n+1}$,
for $n\geq 4$.
\end{prop}

\begin{proof}
As $L_{n+1}^n$ deforms into $L_n^{n-1}$ (see Section \ref{sect_smooth}), all 
Gorenstein genus 1 singularities with at most $n$ branches occur over  $\P(B_n)$.
There cannot be a  proper subcurve with $p_a=1$: if there would be one with
degree $k+1<n+1$, then the cone over the hyperplane section at infinity would
be of type $L_{k+1}^k\vee L_{n-k}^{n-k}$, which is not Gorenstein. For $m=n$ the 
condition (2) in Definition \ref{def:smyth} excludes the case of a proper subcurve
with $p_a=1$.

By the definition of a coarse moduli space there is a map to $\P(B_n)$, which is
bijective on closed points. As  $\P(B_n)$ is normal, it is an isomorphism.
\end{proof}

For a discussion of the identification as stacks over $\Spec \Z[1/6]$, see \cite{LP}.

\section{Elliptic partition curves}
In this section we discuss the deformation theory of general 
partition curves.

\subsection{}
The equations for  the monomial curve with parametrisation
 $z_1=t^{n+1},z_2=t^{n+2}, \dots,z_{n-1}=t^{2n-1},z_n=t^{2n}$
 are 
 \[
 z_i z_j =
 \begin{cases}
 z_1 z_{i+j-1}, &i+j \leq n+1,\\
 z_2 z_{n}, &i+j =n+2,\\
  z_1^2 z_{i+j-n-2} ,&i+ j\geq n+3,
\end{cases}
\]
where $2\leq i \leq j$.
The curve deforms into  other elliptic partition curves of the same 
multiplicity $n+1$. An explicit deformation  is 
\[
 z_i z_j =
 \begin{cases}
 z_1 z_{i+j-1}, &i+j \leq n+1,\\
 z_2 z_{n}, &i+j =n+2,\\
  (z_1^2+a_2 z_n+\dots+a_n z_2 +a_{n+1}z_1) z_{i+j-n-2} ,&i+ j\geq n+3.
\end{cases}
\]
The projection on the $(z_1,z_2)$-plane is given by
$z_2^{n+1}=z_1^{n+2}+a_2z_1^2z_2^{n-1}+\dots+a_{n+1}z_1^{n+1}$
and the factorisation of 
$z_2^{n+1}-a_2z_1^2z_2^{n-1}-\dots-a_{n+1}z_1^{n+1}$ determines the partition of $n+1$.  In particular we find  a non-rational form of $L_{n+1}^n$: 
the (quadratic)  equations, similar to the above ones but with   $z_i z_j =
  z_1 z_{i+j-n-2}$ for $i+j\geq n+3$, 
define $n+1$ lines through the origin, passing through the points
$(\ep,\ep^2,\dots,\ep^n)$ with $\ep$ running through the $(n+1)$-st roots
of unity. 

Knowing the base space for the monomial 	curve implies knowing
the base space for all elliptic partition curves. However, it seems
unfeasible to compute the base space in general.
For the monomial curve of multiplicity 6 the computation of the versal deformation
is described in some detail in \cite{CS}. The base space is just as for $L^5_6$ a cone
 over the Grassmannian $G(2, 5)$, more precisely it is $B_5\times \A^5$. 
 By openness of versality the base space for a partition curve with $r$ branches
 is $B_5\times \A^{6-r}$.  
 We have computed  equations for the base space of the monomial curve of multiplicity 7. The  equations contain a very large number of monomials, and the result is too complicated to be given here. Even for the non-rational form of  seven lines
the equations are rather complicated and not very useful; in particular, the
inductive structure given by Theorem \ref{thm:basetotal} is not visible.

\subsection{}
The same  phenomenon with complicated equations  occurs for the versal deformation of (rational)
partition curves. For $L_n^n$ in the form of the coordinate axes
the result is very easy, see \cite{Sch} or
\cite{St}. The dimension of $T^1$ is $n(n-2)$, but we use 
 $n(n-1)$ deformation parameters $a_{ij}$:
  the coordinate
transformations  $z_i\mapsto z_i-\delta_i$ induce
$a_{ij}\mapsto a_{ij}+\delta_i$, for all $j$. 
The versal family is given by 
\[
F_{ij}=(z_i-a_{ij})(z_j-a_{ji})-
(a_{ik}-a_{ij})(a_{jk}-a_{ji})
\] 
with as equations for the base space
\[
(a_{ik}-a_{ij})(a_{jk}-a_{ji})-(a_{il}-a_{ij})(a_{jl}-a_{ji})=0\;.
\]
For the irrational form of  $L_n^n$, which is the hyperplane section
$z_0=z_n$ of the standard form of the  cone over the ratonal normal curve of degree $n$,
the equations become much more complicated. 
For the monomial curve the computation is
only done up to multiplicity 5 \cite{St}. The quadratic part of the equations of the base
space involves as many variables as the base space of $L_5^5$ and it defines 
a degeneration of that space. 

The (reduced)
base space for $L_n^n$ has been studied by Goryonov and Lando \cite{GL}.
In particular, its degree in $n^{n-3}$.  By a recent result of 
Polishchuk and Rains \cite{PR} this space is Cohen-Macaulay; they identify its 
coordinate ring with an algebra of global sections.

Cohen-Macaulayness does  in general  not hold for the base space
of the monomial partition curve $X_n$.

\begin{thm}
The base space of the versal deformation
of the monomial partition curve $X_n$ with $n\geq 14$ has components
which are not smoothing components; there are components of different
dimensions.
\end{thm}

The idea behind this statement is that the curves  $X_n$ are the most singular 
curves, in the sense that that any curve singularity degenerates to a partition curve
(in a $\delta$-constant degeneration) \cite{Boz}.

Pinkham's examples of non-smoothable curves are among the 
 curves $L_r^n$ consisting of $r$ lines in general position
through the origin in $\A^n$ \cite{Pi, Gr}, with $n+1\leq r \leq \binom {n+1}2$.
One has $\delta(L_r^n)= 2r-n-1$, so the genus of the curve is  $g=r-n$.
We can write $\delta= n+2g-1$.
The general such curve is not smoothable if $r> n+2 + \frac6{n-5}$ (where $n>5$),
or in terms of $g$ if $r> g+5 + \frac6{g-2}$. For $g=4$ and $g=6$ it 
suffices that $r> g+5 $, see e.g. \cite{St-ns}.
The smallest example of a non-smoothable curve is $L_{10}^6$. The dimension
of its base space is $15$. A general smoothable  $L_{10}^6$ has a reducible
base space, with a smoothing component of dimension $20$ and the $15$-dimensional
equisingular component. For the general smoothable $L_r^n$ the smoothing component
has dimension smaller than the number of moduli.

A partition curve with $r=n+g$ branches and $\delta= n+2g-1$ is
$L_n^n\vee gA_2$, where $gA_2$ stands for $A_2\vee \dots\vee A_2$;
it belongs to the partition $(1,\dots,1,2\dots,2)$ of $n+2g$.

\begin{lemma}
The curve $L_{r-1}^n\vee A_2$ deforms into $L_r^n$.
\end{lemma}

\begin{proof}
Let the $r$-th line in $L_r^n$ be  parametrised by
$(a_1t,\dots,a_nt)$. We may suppose that $a_1= 1$. 
Consider the deformation of the parametrisation of 
$L_{r-1}^n\vee A_2\subset \A^n\times \A^2$ , 
where only the paramatrisation of the cusp $A_2$ is
changed: 
$(a_1st,\dots,a_nst,t^2,t^3)$; here $s$ is the deformation parameter. 
As the $r$ lines impose independent conditions
on quadrics, there exists a quadric $q$ in the variables $(z_1,\dots,z_n)$ which vanishes
on the first $r-1$ lines, and restricts to $t^2$ on the line $(a_1t,\dots,a_nt)$.
The coordinate transformation $z_{n+1}\mapsto z_{n+1}-q/s^2$, 
$z_{n+2}\mapsto z_{n+2}-z_1q/s^3$ transforms $(a_1st,\dots,a_nst,t^2,t^3)$
into $(a_1st,\dots,a_nst,0,0)$ and leaves the first $r-1$ lines unchanged. Therefore
the resulting curve is for $s\neq 0$
isomorphic to $L_r^n$.
The deformation is flat as it is $\delta$-constant.
\end{proof}

As the monomial curve $X_{n+2g}$ deforms into $L_n^n\vee gA_2$, the proof of this lemma
shows that it  deforms into $L_{n+g}^n$ for any (general) position of the lines,
including smoothable  $L_{n+g}^n$, which have components of different dimensions.
We note also that $X_{N}$ deforms into $X_{N-1}$. Therefore $X_N$ with 
$N\geq 14$ has a base space with components  of different dimensions.

Such a simple argument is not available for elliptic partition curves. It would be interesting
to know into which type of singularities the Gorenstein monomial curve can deform.

\end{document}